\documentclass[11pt]{article}

\usepackage{amsthm,amsfonts,amssymb,amsmath,color}

\numberwithin{equation}{section}

\renewcommand\d{\partial}
\renewcommand\a{\alpha}
\renewcommand\b{\beta}

\newcommand\N{\mathbb N}\newcommand\Z{\mathbb Z}

\def\g{\gamma}
\def\de{\delta}

\def\O{\Omega}
\def\th{\theta}

\def\vp{\varphi}
\def\epsilon{\varepsilon}
\def\e{\varepsilon}

\newcommand\br{\begin{rem}}
\newcommand\er{\end{rem}}
\newcommand\bp{\begin{pmatrix}}
\newcommand\ep{\end{pmatrix}}
\newcommand\be{\begin{equation}}
\newcommand\ee{\end{equation}}
\newcommand\ba{\begin{equation}\begin{aligned}}
\newcommand\ea{\end{aligned}\end{equation}}

\newcommand\nn{\nonumber}

\newcommand{\CalB}{\mathcal{B}}

\newcommand{\II}{{\mathbb I}}

\newcommand{\SSS}{{\mathbb S}}

\newcommand{\supp}{{\rm supp }}

\newcommand{\vv}{{\mathbf v}}

\newcommand{\vr}{\varrho}

\newcommand{\vu}{\vc{u}}
\newcommand{\vf}{\vc{f}}
\newcommand{\vg}{\vc{g}}
\newcommand{\vc}[1]{{\bf #1}}
\newcommand{\Div}{{\rm div\,}}
\newcommand{\Grad}{\nabla_x}

\newcommand{\dx}{{\rm d} {x}}

\newcommand{\intO}[1]{\int_{\O} #1 \ \dx}
\newcommand{\intOe}[1]{\int_{\O_\e} #1 \ \dx}

\newcommand{\dive}{{\rm div\,}}

\newtheorem{defi}{Definition}[section]
\newtheorem{theorem}[defi]{Theorem}
\newtheorem{proposition}[defi]{Proposition}
\newtheorem{lemma}[defi]{Lemma}

\newtheorem{remark}[defi]{Remark}

\numberwithin{equation}{section}
\hoffset - 1 cm

\begin{document}

\title{Homogenization problems for the compressible Navier-Stokes system in 2D perforated domains}

\author{ \v S\' arka Ne\v casov\' a \footnote{Institute of Mathematics, Czech Academy of Sciences, \v Zitn\' a 25, 115 67 Praha 1, Czech Republic, {\tt matus@math.cas.cz.}} \and  Jiaojiao Pan\footnote{Corresponding author, Department of Mathematics, Nanjing University, 22 Hankou Road, Gulou District, Nanjing 210093, China, {\tt panjiaojiao@smail.nju.edu.cn.}}}

\date{}

\maketitle

\begin{abstract}

In this paper, we study the homogenization problems for the stationary compressible Navier-Stokes system in a bounded 2D domain, where the domain is perforated with very tiny holes (or obstacles) whose diameters are much smaller than their mutual distances. We obtain that the process of homogenization doesn't change the motion of the fluids. From another point of view, we obtain the same system of equations in the asymptotic limit. It is the first result of homogenization problem in the compressible case in 2 dimension.

\end{abstract}



\renewcommand{\refname}{References}

\section{Introduction}
\label{Introduction}
Homogenization of {color{blue} the }Newtonian fluid in physical domains perforated by a large number of tiny holes plays an important role in fluid mechanics and has gain lots of interest. Viscous fluid flows passing a great many fixed solid obstacles is a situation frequently occurring in real applications referring to \cite{book-hom}. Based on these applications, the models like stationary, viscous fluid flows represented by the the standard Stokes or Navier-Stokes system of equations in porous medium could be of vital importance.

The typical diameter and mutual distance of these holes become the main factors in the asymptotic behavior of fluid flows
in the regime where the number of holes tends to infinity and their size tends to zero. With the increasing number of holes, the limit motion of fluid flow approaches a state governed by certain {\em homogenized} equations which are homogeneous in form of different cases (without obstacles).

Allaire \cite{ALL-NS1}, \cite{ALL-NS2} (or earlier results by Tartar \cite{Tartar1}) provided a systematical study of Stokes and/or Navier-Stokes system for three different circumstances where the holes are periodically distributed with different size. Roughly speaking, the asymptotic limit behavior is governed by Darcy's law when the holes are of supercritical size;
when the holes are of the critical size, the asymptotic limit behavior gives rise to Brinkman's law; the subcritical size of holes makes no differences in the motion of the asymptotic limit - the limit problem coincides with the original system of equations.

Moreover, relevant results has been extended to the evolutionary (time-dependent) incompressible Navire-Stokes system by Mikeli\'{c} \cite{Mik} and Allaire \cite{ALL-3}, and more recently in \cite{FeNaNe}. For the evolutionary barotropic (compressible) Navier-Stokes system, Masmoudi \cite{Mas-Hom} obtained that the homogenization
limit was governed by porous medium equation (Darcy's law) where the diameter of the holes is comparable to their mutual distance (critical size), and similar results for the full Navier-Stokes-Fourier system were obtained in \cite{FNT-Hom}.  In \cite{ DFL, FL1, Lu-Schwarz18}, the authors considered the case of small holes for the compressible Navier-Stokes equations and in \cite{Lu-Pokorny}, steady compressible Navier–Stokes–Fourier system is considered, homogenized equations mentioned above remain the same asymptotic limit as the original ones. Let us also mention the case when the hole is the moving rigid body. In the case of the compressible fluid with the moving rigid body the homogenization was study by Bravin and Ne\v casov\' a \cite{BN}.

\subsection{Problem formulation}
\label{Problem formulation}

Similar to the case in \cite{DFL} and \cite{FL1}, where the homogenization of 3D steady compressible Navier-Stokes equations is considered, here we concentrated on homogenization of the compressible (isentropic) stationary Navier-Stokes system in two spatial dimensions of the subcritical case and we show that the asymptotic limit coincides with the original one, where $\e$ denotes the mutual distance between the holes and the diameter of the holes is taken as $a_{\e}= e^{-\sigma\varepsilon^{-\alpha}}$ with $\a>2$.
Considering a bounded domain $\Omega \subset R^2$ of class $C^2$, we introduce a family of $\e$-dependent \emph{perforated domains}
$\{ \Omega_\e \}_{\e > 0}$,
\be\label{1.1}
\Omega_\e = \Omega \setminus \bigcup_{k\in K_\e} T_{\e, k},\ \ K_\e:=\{k|\ \e \bar{C}_k\subset\O\}
\ee
where the sets $T_{\e,k}$ represent \emph{holes} or \emph{obstacles}. We suppose the following property concerning the distribution of the holes:
 \be\label{1.2}
T_{\e,k} := x_{\e,k} + a_\e T_k \subset \subset B(x_{\e,k}, b_0a_{\e} ) \subset \subset \e C_k\subset\O,
\ee
with
 \be\label{1.3}
C_k := (0,1)^2+k,\ \ k\in \mathbb{Z}^2,\ a_{\e}= e^{-\sigma\varepsilon^{-\alpha}}\ \mbox{for} \ \ \alpha>2.
\ee
Here $x_{\e,k}\in T_{\e,k}, \ b_0>0$ and $\sigma$ is positive constant independent of $\varepsilon$, for each $k$,  $T_{k}\subset R^2$ is a simply connected bounded domain of class $C^2$, $B(x,r)$ denotes the open ball centered at $x$ with radius $r$ in $R^2$.
The diameter of each $T_{\e,k}$ is of size
$O(a_\e)$ and their mutual distance  is $O(\e)$, where their total number $|K_\e|$ can be estimated as
\be\label{number-holes}
|K_{\e}| \leq  \frac{|\Omega|}{\e^2}(1+o(1)).\nn
\ee
For convenience, we use the symbol $L^{q}_0(\Omega)$ to denote the space of functions in $L^q(\Omega)$ with zero integral mean
\be\label{1.4}
L^q_0(\Omega):=\left\{f\in L^q(\Omega): \, \int_{\O} f\,\dx=0\right\}.
\ee
Then, we consider the following stationary (compressible) \emph{Navier-Stokes system in $\Omega_{\varepsilon}$}
\be\label{1.5}
\dive (\vr \vu) = 0,
\ee
\be\label{1.6}
 \dive (\vr \vu \otimes \vu) +\nabla p(\vr) =
\dive \SSS(\nabla \vu)+\vr \vf+\vg,
\ee
\be\label{1.7}
\SSS (\nabla \vu) = \mu \left( \nabla \vu + \nabla^t \vu - \frac{2}{3} (\dive \vu) \II \right) + \eta (\dive \vu) \II,\ \mu > 0,\ \eta \geq 0.
\ee
Here $\vr$ is the fluid mass density, $\vu$ is the velocity field, $p=p(\vr)$ is the pressure, $\SSS(\nabla \vu)$
stands for the Newtonian viscous stress tensor with constant viscosity coefficients $\mu, \eta$.
\par In the spatial domain $\O_\e$, the system is supplemented with the standard no-slip boundary condition
\be\label{1.8}
\vu=0\quad \mbox{on}\ \partial\Omega_\e.
\ee

For the sake of simplicity, we concentrate on the
\emph{isentropic} pressure-density state relationship
\be\label{1.9}
p(\vr)=a\vr^\gamma,\ a > 0,
\ee
with the adiabatic exponent $\g$, which will be specified as follows.

The motion of the fluid is driven by the volume force $\vf$ and nonvolume force $\vg$, defined on the whole domain $\Omega$ and independent of $\e$,
those are supposed to be uniformly bounded,
\be\label{1.10}
\|\vf\|_{L^\infty(\O; R^2)}+\|\vg\|_{L^\infty(\O; R^2)}\leq C<\infty.
\ee
Specifically, we use the symbol $C$ to denote a generic bounded constant that may vary from line to line in the following contents but it is independent of the parameters of the problem, in particular of $\e$. Furthermore, we use the symbol $\tilde{h}$ to denote the $zero-extension$ of $h$ in $R^{2}$, which means
\be\label{1.11}
\tilde{h}=h \ \ \mbox{in} \ \ \Omega_{\e}, \ \ \ \tilde{h}=0 \ \ \mbox{in} \ \ R^{2}\setminus\Omega_{\e}.
\ee

To be consistent with its physical interpretation,
the density $\vr$ is non-negative and we fix the total mass of the fluid to be
\be\label{1.12}
\displaystyle 0<\inf_{0<\e<1} M_\e\leq M_\e = \int_{\Omega_\e} \vr \ \dx\leq\sup_{0<\e<1} M_\e<\infty .
\ee

In particular, the $Restriction \ operator$ introduced by Allaire \cite{ALL-NS1} can be used in a compatible way in $2D$ to construct the inverse of the
divergence - the so-called Bogovskii's operator (see Bogovskii \cite{bog}, Galdi \cite[Chapter 3]{Galdi}).

 The paper is organized as follows. In Section \ref{Problem formulation} to Section \ref{Main results}, we introduce the formulation of the problem, the definition of weak solutions and state our main results. Then in section \ref{Preliminaries}, we introduce the Restriction operator and construct the inverse of the divergence - a Bogovskii's type operator which plays a crucial role in the proof of the uniform bound to the solution [$\rho$,$u$]. After that, uniform estimates are obtained via this operator in Section \ref{Uniform bounds} to identify the asymptotic limit for the Navier-Stokes system in perforated domains. In Section \ref{Equations in a fixed domain} to Section \ref{Convergence of pressure term - end of the proof}, we give the convergence (or homogenization) process of 2D compressible Navier-Stokes system in perforated domains, which shows the homogenization process for the stationary compressible Navier-Stokes equations in a perforated domain is not affected by obstacles and the limit problem coincides with the original one. Finally, we obtain the main results listed in Section \ref{Main results}.

\subsection{Weak solutions}
\label{Weak solutions}

We recall the definition of finite energy weak solutions to \eqref{1.5}-\eqref{1.8}, see \cite[Definition 4.1]{N-book}.

\begin{defi}\label{Definition 1.1}
A couple of functions $[\vr, \vu]$ is said to be a \emph{finite energy weak solution} of the Navier-Stokes system \eqref{1.5}-\eqref{1.7} supplemented with the conditions \eqref{1.8}-\eqref{1.12} in $\O_\e$ provided:
\be\label{1.13}
 \vr \geq 0~ \mbox{ a.e. in}~ \O_\e, \int_{\O_\e}\vr \ \dx = M_\e,\ \vr \in  L^{\beta(\gamma)}(\O_\e) \ \mbox{for some} ~\beta(\gamma)  > \gamma, \ \vu \in  W_0^{1,2}(\O_\e; R^2);
\ee
 for any test functions   $\psi \in C^\infty( \overline{\O}_\e)$   and $\varphi \in C_c^\infty( {\O_\e}; R^2)$:
\be\label{1.14}
\intOe{ \vr \vu \cdot \nabla \psi  }  =0;
\ee
\be\label{1.15}
 \intOe{ \vr \vu \otimes \vu : \nabla \varphi+ p(\vr) \dive \varphi-\SSS(\nabla \vu) : \nabla \varphi +(\vr \vf+\vg)\cdot \varphi}=0;
\ee

 \item and the \emph{energy inequality}
\be\label{1.16}
\intOe{ \SSS(\nabla \vu) : \nabla \vu } \leq \intOe{ (\vr\vf+\vg)\cdot \vu }
\ee
holds.
\end{defi}

\begin{remark}
A \emph{finite energy weak solution} of the Navier-Stokes system \eqref{1.5}-\eqref{1.7} in $\O$ is similar to it in $\O_\e$.
\end{remark}

\begin{defi}\label{Definition 1.3} A finite energy weak solution $[\vr, \vu]$ (see \cite{N-book}) is said to be a \emph{renormalized weak solution} if
\be\label{1.17}
\int_{R^2}{ b(\vr)\vu \cdot \Grad \psi + \Big( b( \vr) - b'(\vr) \vr \Big) \dive \vu\ \psi }\ \dx = 0
\ee
for any $\psi \in  C_c^\infty(R^2)$, where $[\vr, \vu]$ were extended to be zero outside $\O_\e$,
and any $b\in C^0([0,\infty))\cap C^1((0,\infty))$ such that
\be\label{1.18}
b'(s)\leq c\, s^{-\lambda_0} \ \mbox{for}\  s\in (0,1],\quad \ b'(s)\leq c\,s^{\lambda_1} \ \mbox{for}\  s\in [1,\infty),
\ee
with
\be\label{1.19}
c>0,\quad \lambda_0<1, \quad -1<\lambda_1 \leq \frac{\beta(\gamma)}{2}-1.
\ee
\end{defi}

\begin{lemma}\label{Lemma 1.4}
By DiPerna-Lions' transport theory (see {\rm \cite[Section II.3]{DiP-L}} and the modification in {\rm \cite[Lemma 3.3]{N-book}}), for any $r\in L^\beta(\Omega),~\beta\geq 2,\ \vv\in W^{1,2}_0(\Omega)$, where $\Omega\subset  R^2$ is a bounded domain of class $C^2$, such that
\begin{equation} \label{1.20}
\dive (r\vv)=0 \quad \mbox{in}\quad \mathcal{D}'(\Omega),
\end{equation}
the renormalized equation
\be\label{1.21}
 \dive\big(b(r)\vv\big)+\big(r b'(r) -b(r)\big)\dive \vv=0, \quad \mbox{holds in}\ \mathcal{D}'(R^2),
\ee
for any $b\in C^0([0,\infty))\cap C^1((0,\infty))$ satisfying \eqref{1.18}-\eqref{1.19} provided $r$ and $\vv$ have been extended to be zero outside $\Omega$.
\end{lemma}

\begin{remark}\label{Remark 1.5}
Recall that the relevant values of the adiabatic exponent are $1 \leq \gamma \leq 5/3$, where the case $\gamma = 1$ corresponds to the isothermal case while $\gamma = 5/3$ is the adiabatic exponent of the monoatomic gas. Based on energy type arguments combined with the refined pressure estimates, Lions \cite{Lions-C} proved the existence of weak solutions in the range $\gamma > 5/3$ in 3D. Lions' theory for the existence of weak solutions to homogeneous Dirichlet (no-slip) boundary conditions has been extended to the physical range $\gamma \leq 5/3$ by several authors, see B\v rezina and Novotn\' y \cite{BREZNOV}, Plotnikov and Sokolowski \cite{PloSokbook} for $\g>3/2$, Frehse, Steinhauer and Waigant \cite{FSW} for $\gamma > 4/3$, Plotnikov and Weigant \cite{PW} for $\g>1$. While in two spatial dimensions with $\g>1$, Lions (\cite{Lions-C}, Chapter 6) has given the existence of the weak solution $[\vr_\e,\vu_\e] \in [L^{2\g}(\O_\e)]\times[H_0^1(\O_\e)]^2$, which establishes the following uniform norm estimates for the solution $[\rho,u]$ in $\O$. Then, we need to show that the asymptotic limit of solutions $[\vr_\e, \vu_\e]$ of the compressible Navier-Stokes system (\ref{1.5})-(\ref{1.8}) in $\O_\e$ coincides with a solution of the same system on the homogeneous domain $\Omega$.
\end{remark}

\subsection{Main results}
\label{Main results}
In this paper, we consider the case of optimal adiabatic exponent $\g>1$ in the pressure law (\ref{1.9}), which is also an innovation in this article.

\begin{theorem}\label{Theorem 1.6}
Suppose the conditions \eqref{1.9}, \eqref{1.10}, \eqref{1.12} are satisfied. Let $\g>1$ and  $\a>2$ be given and $[\vr_\e,\vu_\e]_{0<\e<1}$ be a family of finite energy weak solutions to \eqref{1.5}-\eqref{1.8} in $\O_\e$, where \vf \ and \vg \ satisfy \eqref{1.10}. Then we have uniform estimates
\be\label{1.22}
\sup_{0<\e<1}\left(\|\vr_\e\|_{L^{2\g}(\O_\e)}+\|\vu_\e\|_{W_0^{1,2}(\O_\e; R^2)}\right)\leq C<\infty.
\ee
Moreover, up to a substraction of subsequence, the zero-extentsions $[\tilde{\vr_\e}, \tilde{\vu}_\e]$ to outside $\Omega_\e$ satisfy
\be\label{1.23}
\tilde{\vr}_\e  \to \vr \ \mbox{weakly in} \ L^{2\g}(\O),\quad  \tilde{\vu}_\e \to \vu \   \mbox{weakly in}\  W_{0}^{1,2}(\O; R^2),
\ee
where $[\vr, \vu]$ is a finite energy weak solution to the same system of equations \eqref{1.5}-\eqref{1.8} in the limit domain $\O$.

\end{theorem}
\subsection{Preliminaries}

\label{Preliminaries}
\subsubsection{Bogovskii's operator}
Our principal result concerns the construction of the inverse of the divergence operator on the family of perforated domains $\{ \O_\e \}_{\e > 0}$.
\begin{proposition}\label{Proposition 1.7} Let $\{ \O_\e \}_{\e > 0}$ be a family of domains with the properties specified in {\rm Section \ref{Problem formulation}}. Then there exists a linear operator
\[
\CalB_\e : L_0^{2}(\O_\e) \to W_0^{1,2}(\O_\e; R^2),
\]
such that for any $f\in L_0^{2}(\O_\e)$,
\ba\label{1.24}
&\dive(\CalB_\e(f)) =f \ \mbox{in} \ \O_\e,\\
&\|\CalB_\e(f)\|_{W_0^{1,2}(\O_\e; R^2)}\leq C\|f\|_{L^2(\O_\e)},
\ea
for some constant $C$ independent of $\e$.
\end{proposition}

The existence of Bogovskii's operator on a \emph{fixed} Lipschitz domain has been shown by several authors, such as Galdi \cite{Galdi} or
Acosta, Dur\'an and Muschietti \cite{ADM} and  Diening, R{\r{u}}\v{z}i\v{c}ka and Schumacher \cite{DRS}, especially by Bogovskii \cite{bog},  Pileckas and Kapitanskii \cite{PK}. Here, we need to establish the uniform estimate (\ref{1.24}), where the constant $C$ is independent of $\e$. For the sake of readers, we briefly give a proof of Proposition 1.7. Referring to \cite{DFL}  and \cite{FL1}, such a Bogovskii's operator is established in 3D case. Considering the significance of the Restriction operator in proving the existence of Bogovskii's operator and therefore, we introduce the following Restriction operator $R_{\e}$ analogously by employing the Restriction operator constructed by Allaire (\cite{ALL-NS1}, Section 2.2).

\begin{lemma}\label{Lemma 1.8}
 For $\Omega_{\e}$ defined in Section \ref{Problem formulation}, there exists a linear bounded Restriction operator $R_{\e}: W_{0}^{1,2}(\Omega;R^{2})\rightarrow W_{0}^{1,2}(\Omega_\e;R^{2})$ such that

\be\label{1.25}
\mathbf{u}\in W_{0}^{1,2}(\O_\e; R^2)\Longrightarrow R_{\e}(\tilde{\mathbf{u}})=\mathbf{u}\ \ \mbox{in} \ \ \ \Omega_{\e}
\ee

\be\label{1.26}
\dive\mathbf{u}=0\ \ \mbox{in}\ \ \Omega\Longrightarrow \dive R_{\e}(\mathbf{u})=0\ \ \mbox{in} \ \ \ \Omega_{\e}
\ee

\be\label{1.27}
\|R_{\e}(\mathbf{u})\|_{W_0^{1,2}(\O_\e; R^2)}\leq C\|\mathbf{u}\|_{W_0^{1,2}(\O; R^2)},\ \ C\ \mbox{independent of}\ \e,
\ee
\end{lemma}
This Restriction operator is constructed in the proof of [\cite{ALL-NS1}, Lemma 2.1]  in the following way

\be\label{1.28}
 B(x_{k}, b_1\e)\subset\e C_k,\  \bar{B}_{\e,k}=\bar{B}(x_{k},b_0a_\e)\subset B(x_{k}, b_1\e).
\ee
Let us introduce the following decomposition of the cube $\e C$ with $k\in K_\e$:
\be\label{1.29}
 \e \bar{C}_k=T_{\e, n}\bigcup \bar{E}_{\e, k}\bigcup \bar{F}_{\e, k}\ \ \mbox{with} \ \ E_{\e, k}:=B(x_{k}, b_1\e)\setminus T_{\e, k},\ F_{\e, k}:=(\e C_k)\setminus B(x_{k}, b_1\e),
\ee
where $b_1>0$. For any $\mathbf{u}\in W_{0}^{1,2}(\O; R^2)$, we can define $R_{\e}$ by
\begin{eqnarray}
\label{1.30} \left\{
   \begin{array}{ll}
    {R_{\e}(\mathbf{u})}=\mathbf{u} \ \ \mbox{on} \ \ \e C_k\bigcap\Omega,\ \ \mbox{for}\ \  k \notin K_\e
 \\
   {R_{\e}(\mathbf{u})}=\mathbf{u} \ \ \mbox{on}\ \  F_{\e, k},\ \ {R_{\e}(\mathbf{u})}=0 \ \ \mbox{on}\ \  T_{\e, k},\ \ {R_{\e}(\mathbf{u})}=\mathbf{v}_{\e,k} \ \ \mbox{in}\ \  E_{\e, k},\ \ \mbox{for}\ \  k\in K_\e.
\end{array}
\right.
\end{eqnarray}
where $\mathbf{v}_{\e,k}\in W_0^{1,2}(\e C_k; R^2)$ satisfies
\be\label{1.31}
\nabla p_{\e,k}-\Delta\mathbf{v}_{\e,k}=-\Delta\mathbf{u} \ \ \mbox{in}\ \  E_{\e, k}
\ee
\be\label{1.32}
\mbox{div}\mathbf{v}_{\e,k}=\mbox{div}\mathbf{u}+\frac{1}{|E_{\e,k}|}\int_{ T_{\e, k}}\mbox{div}_{x}\mathbf{u} \ dx \ \ \mbox{in} \ \ E_{\e, k},
\ee

\be\label{1.33}
\mathbf{v}_{\e,k}=\mathbf{u}\ \ \mbox{on}\ \ \partial E_{\e, k}-\partial T_{\e, k},\ \ \mathbf{v}_{\e,k}=0 \ \ \mbox{on} \ \ \ \partial T_{\e, k}.
\ee

\begin{proof}[\bf Proof of Proposition 1.7.]
For $f\in L_{0}^{2}(\O_\e)$, we consider the following zero-extension of $f$
\be\label{1.34}
\bar{f}=f\ \ \mbox{in}\ \ \O_\e,\ \ \bar{f}=0 \ \ \mbox{on} \ \ \O\backslash\O_\e= \Omega \setminus \bigcup_{k\in K_\e}T_{\e,k}.
\ee
By classical Bogovskii's operator (see Bogovskii \cite{bog}) defined on domain $\O$, we obtain $\mathbf{u}:=\mathcal{B}(f)\in W_{0}^{1,2}(\O; R^2)$ satisfying

\be\label{1.35}
\mbox{div}\mathbf{u}=f\ \ \mbox{in}\ \ \O_\e,\ \ \|\mathbf{u}\|_{W_0^{1,2}(\O; R^2)}\leq C\|\tilde{f}\|_{L^{2}(\O; R^2)}=C\|f\|_{L^{2}(\O_\e; R^2)}
\ee
where $C$ depends only on $\O$. Furthermore, by (\ref{1.34}) we have
\be\label{1.36}
\mbox{div}\mathbf{u}=\tilde{f}=0\ \ \mbox{in}\ \ T_{\e, k}.
\ee
Applying the Restriction operator constructed in (\ref{1.30})-(\ref{1.33}), we obtain
 \be\label{1.37}
\mbox{div}\mathbf{v_{\e,k}}=\mbox{div}\mathbf{u}=\tilde{f}=0\ \ \mbox{on}\ \ E_{\e, k},
\ee
whenever $\textbf{u}$ satisfies (\ref{1.35}). Moreover, we have ${R_{\e}(\mathbf{u})}=\mathbf{u}$ in $ \Omega \setminus \bigcup_{k\in K_\e}E_{\e,k}$. Combined with (\ref{1.30}) and (\ref{1.37}), we conclude that
\be\label{1.38}
\mbox{div}R_\e(\mathbf{u})=f \ \ \mbox{in}\ \ \O_\e.
\ee
Now we define
\be\label{1.39}
\mathcal{B}_\e(f):={R_{\e}(\mathbf{u})}={R_{\e}(\mathcal{B}(\tilde{f}))}.
\ee
where $\mathcal{B}$ is the classical Bogovskii's operator on $\O$. It's easy to check that
\be\label{1.40}
\|\mathcal{B}_\e(f)\|_{W_0^{1,2}(\O_\e; R^2)}:=\|{R_{\e}(\mathbf{u})}\|_{W_0^{1,2}(\O_\e; R^2)}\leq C \|\textbf{u}\|_{W_0^{1,2}(\O; R^2)}=C\|\mathcal{B}(f)\|_{W_0^{1,2}(\O; R^2)}\leq C\|\tilde{f}\|_{L^{2}(\O)}=C\|f\|_{L^{2}(\O_\e)}.
\ee
Thus we proved Proposition 1.7.
\end{proof}
\section{Proof of Theorem 1.6}
This section deals with the proof of Theorem 1.6. By Lemma \ref{Lemma 1.4},  let us stress that the solution $[\vr_\e,\vu_\e]$ is also a renormalized weak solution as follows.
\begin{lemma}\label{Lemma 2.1}
We have
\ba
\dive\big(b(\tilde\vr_\e)\tilde\vu_\e\big)+\big(\tilde\vr_\e b'(\tilde\vr_\e) -b(\tilde\vr_\e)\big)\dive \tilde\vu_\e=0, \quad \mbox{in}\ \mathcal{D}'(R^2),\nn
\ea
for any $b\in C^0([0,\infty))\cap C^1((0,\infty))$ satisfying \eqref{1.18}-\eqref{1.19}, where $[\tilde \vr_\e, \tilde \vu_\e]$ denotes the functions $[\vr_\e, \vu_\e]$ extended
to be zero outside $\O_\e$.
\end{lemma}
\subsection{Uniform bounds}
\label{Uniform bounds}

As shown in Lions (\cite{Lions-C}, Chapter 6) with $N=2, \g>1$, we have the existence of solution $[\vr_\e,\vu_\e] \in [L^{2\g}(\O_\e)]\times[H_0^1(\O_\e)]^2$ for any fixed $\e$. However, the classical estimates of their norms depend on the Lipschitz character of domain $\O_\e$, which goes to infinity as $\e\rightarrow 0$. In order to show the uniform estimates \eqref{1.22}, we employ the uniform Bogovskii type operator $\CalB_\e$  constructed in Section \ref{Preliminaries} to establish the independence of $\e\in(0,1)$ in \eqref{1.22}.

By the  Korn's inequality and the H\"older's inequality, the energy inequality \eqref{1.16} implies
\ba\label{2.1}
 \|\nabla  \vu_\e \|^2_{L^2(\O_\e; R^{2\times2})} &\leq C\, \Big(\|\vf\|_{L^\infty(\O_\e; R^2)} \|\vr_\e \|_{L^{s}(\O_\e)} \|\vu_\e \|_{L^t(\O_\e; R^2)}\\
 &\qquad+\|\vg\|_{L^\infty(\O_\e; R^2)}\|\vu_\e \|_{L^1(\O_\e; R^2)}\Big),
\ea
where $\frac{1}{s}+\frac{1}{t}=1$ and $s\rightarrow 1+, \ t\rightarrow \infty-$.
\begin{remark}
In Section \ref{Uniform bounds} and Section \ref{Equations in a fixed domain}, we frequently use $\infty-$ to denote a bounded real number arbitrarily close to positive infinity and $x+(x-)$ denotes the real number arbitrarily close to $x$ on the right (left) side.
\end{remark}
Since $\vu_\e\in W_0^{1,2}(\O_\e; R^2)$ has zero trace on the boundary, the Sobolev embedding  and the Poincar\'{e} inequality imply
\be\label{2.2}
\|\vu_\e\|_{L^t(\O_\e; R^2)}\leq\|\vu_\e\|_{W^{1,2}_{0}(\O_\e; R^2)}\leq C\, \ \|\nabla \vu_\e\|_{L^2(\O_\e; R^{2\times2})},
\ee
\be\label{2.3}
\|\vu_\e\|_{L^1(\O_\e; R^2)}\leq\|\vu_\e\|_{W^{1,2}_{0}(\O_\e; R^2)}\leq C\, \ \|\nabla \vu_\e\|_{L^2(\O_\e; R^{2\times2})},
\ee
for some constant $C$ independent of the domain $\O_\e$.

By the above estimates \eqref{2.1}-\eqref{2.3}, we deduce
\ba\label{2.4}
\|\nabla  \vu_\e \|_{L^2(\O_\e; R^{2\times2})}+\|\vu_\e\|_{L^t(\O_\e; R^2)} &\leq C\, \left(\|\vf\|_{L^\infty(\O_\e; R^2)} \|\vr_\e \|_{L^{s}(\O_\e)} +\|\vg\|_{L^\infty(\O_\e; R^2)}\right) \\
&\leq C\, \left( \|\vr_\e \|_{L^{s}(\O_\e)} +1\right).
\ea

\smallskip

Let $\mathcal{B}_\e$ be the  operator introduced in Proposition \ref{Proposition 1.7}, we introduce the test function
\be\label{2.5}
\varphi:=\mathcal{B}_\e\left(\vr_\e^{\gamma}-\frac{1}{|\O_\e|}\intOe {\vr_\e^{\gamma}}\right).
\ee

By Remark \ref{Remark 1.5}, we notice that
\be\label{2.6}
\vr_\e \in L^{2\gamma}(\O_\e),\quad \vu_\e \in W_{0}^{1,2}(\O_\e; R^2), \ \ \mbox{for any fixed}\ \ \e.
\ee

Then by Proposition \ref{Proposition 1.7} and \eqref{2.5}, we have
\ba\label{2.7}
&\dive \varphi = \vr_\e^{\gamma}-\frac{1}{|\Omega_\e|}\intOe {\vr_\e^{\gamma}}\quad \mbox{in}\ \O_\e
\ea
and
\ba\label{2.8}
\|\varphi\|_{W^{1,2}_0(\O_\e; R^2)}
&\leq C\|\vr_\e^{\gamma}-\frac{1}{|\Omega_\e|}\intOe {\vr_\e^{\gamma}}\|_{L^{2}(\O_\e)}\\
&\leq C\, \left(\|\vr_\e^{\gamma}\|_{L^{2}(\O_\e)}+ \|\vr_\e^{\g} \|_{L^1(\O_\e)}\right) \\
&\leq C\, \|\vr_\e\|^{{\g}}_{L^{2\gamma}(\O_\e)}.
\ea

Taking $\vp$ as a test function in the weak formulation of the momentum equation \eqref{1.15} gives
\be\label{2.9}
\intOe{p(\vr_\e)\vr_\e^{\g}} = \sum_{ j=1}^4 I_j
\ee
with
\ba\label{Ij}
&I_1:=\intOe{p(\vr_\e)} \ \frac{1}{|\O_\e|}\intOe {\vr_\e^{\gamma}}, \\
&I_2:=\intOe{\mu \nabla \vu_\e:\nabla \varphi}+\intOe{\left(\frac{\mu}{3} + \eta \right) \dive \vu_\e :\dive \varphi},\\
&I_3:=-\intOe{\vr_\e \vu_\e\otimes\vu_\e:\nabla \varphi},\\
& I_4:=-\intOe{(\vr_\e \vf +\vg)\cdot\varphi}.\nn
\ea
For $I_1$:
\ba\label{2.10}
&I_1:=\intOe{p(\vr_\e)} \ \frac{1}{|\O_\e|}\intOe {\vr_\e^{\gamma}}=\frac{a}{|\O_\e|}\,\|\vr_\e\|_{L^{\g}(\O_\e)}^{2\g}\\
&\leq \frac{a}{|\O_\e|}\,\|\vr_\e\|_{L^{1}(\O_\e)}^{2\g\th_1} \|\vr_\e\|_{L^{2\g}(\O_\e)}^{2\g(1-\th_1)}=\frac{aM_\e^{2\g\th_1}}{|\O_\e|} \|\vr_\e\|_{L^{2\g}(\O_\e)}^{2\g(1-\th_1)}, 
\ea
where we used \eqref{1.12}, Young's inequality, and interpolations between Lebesgue spaces. $M_\e$ is the total mass and the number $\th_1$ satisfies
\ba\label{2.11}
0<\th_1<1\ \  \mbox{s.t.}\ \  \frac{1}{\g}=\frac{\th_1}{1}+\frac{1-\th_1}{2\g}.
\ea

For $I_2$:
\ba\label{2.12}
I_2 &\leq  C\, \|\nabla \vu_\e\|_{L^2(\O_\e; R^{2\times2})}\|\nabla\varphi\|_{L^2(\O_\e; R^{2\times2})}\leq C\, \left( \|\vr_\e \|_{L^{s}(\O_\e)} +1\right)\|\vr_\e\|_{L^{2\g}(\O_\e)}^{\g} \\
&\leq C\, \left( \|\vr_\e \|_{L^{1}(\O_\e)}^{\th_2}\|\vr_\e \|_{L^{2\g}(\O_\e)}^{1-\th_2} +1\right)\|\vr_\e\|_{L^{2\g}(\O_\e)}^{{\g}}\\
&\leq C\, \left( M_\e^{\th_2}\|\vr_\e \|_{L^{2\g}(\O_\e)}^{\g+1-\th_2}+\|\vr_\e \|_{L^{2\g}(\O_\e)}^{\g}  \right),
\ea
where we used \eqref{1.12}, \eqref{2.4} and \eqref{2.8}. The number $0<\th_2<1$ is determined by
\be\label{2.13}
 \frac{1}{s}=\frac{\th_2}{1}+\frac{1-\th_2}{2\g},\ \ \frac{1}{s}+\frac{1}{t}=1,\ \ s\rightarrow 1+, \ t\rightarrow \infty-
\ee
which implies $\th_2=\frac{\frac{1}{s}-\frac{1}{2\g}}{1-\frac{1}{2\g}}\rightarrow 1^(1-\th_2\rightarrow0+)$ for $s\rightarrow1+$ and any $\g\in(1,+\infty)$.
\par For $I_3$:
\ba\label{2.14}
I_3 &=-\intOe{\vr_\e \vu_\e\otimes\vu_\e:\nabla \varphi}\\
&\leq  C\, \|\vr_\e\|_{L^{s_{1}}(\O_\e)} \| \vu_\e\|_{L^{t_{1}}(\O_\e; R^2)}^2 \|\nabla\varphi\|_{L^{2}(\O_\e; R^{2\times2})}\\
&\leq C\, \|\vr_\e\|_{L^{1}(\O_\e)}^{\th_3}\|\vr_\e\|_{L^{2\g}(\O_\e)}^{1-\th_3}(\|\vr_\e\|_{L^{s}(\O_\e)}^{2}+1)\|\vr_\e\|_{L^{2\g}(\O_\e)}^{\g}\\
&\leq C\, \|\vr_\e\|_{L^{1}(\O_\e)}^{\th_3}\|\vr_\e\|_{L^{2\g}(\O_\e)}^{1-\th_3}(\|\vr_\e\|_{L^{1}(\O_\e)}^{2\th_2}\|\vr_\e\|_{L^{2\g}(\O_\e)}^{2(1-\th_2)}+1)\|\vr_\e\|_{L^{2\g}(\O_\e)}^{\g}\\
&\leq C\, M_{\e}^{\th_3+2\th_2}\|\vr_\e\|_{L^{2\g}(\O_\e)}^{\g+(1-\th_3)+2(1-\th_2)}+C\, M_{\e}^{\th_3}\|\vr_\e\|_{L^{2\g}(\O_\e)}^{\g+(1-\th_3)} ,
\ea
where the estimates are similar to $I_1,I_2 $ and
\be\label{2.15}
0<\th_2, \th_3<1 \quad  \mbox{s.t.}\ \ \frac{1}{s}=\frac{\th_2}{1}+\frac{1-\th_2}{2\g} \ \ \mbox{and} \ \ \ \frac{1}{s_1}=\frac{\th_3}{1}+\frac{1-\th_3}{2\g}.
\ee
This implies
\be\label{2.16}
\th_3=\frac{\frac{1}{s_1}-\frac{1}{2\g}}{1-\frac{1}{2\g}}\rightarrow 0+(1-\th_3\rightarrow1-)\ \ \mbox{as} \ \ s_1\rightarrow2+,\ \ \g\rightarrow1+.
\ee

For $I_4$:
\ba\label{2.17}
I_4&=-\int_{\O_\e}{(\vr_\e \vf +\vg)\cdot\varphi}\\
&\leq C\, \|\vr_\e\|_{L^{s}(\O_\e)}\|\vf\|_{L^\infty(\O_\e; R^2)}\| \vp\|_{L^t(\O_\e; R^2)}+\|\vf\|_{L^\infty(\O_\e; R^2)}\| \vp\|_{L^2(\O_\e; R^2)}\\
&\leq C\,(1+\|\vr_\e\|_{L^{s}(\O_\e)})\|\vp\|_{W^{1,2}(\O_\e; R^2)}\\
&\leq C(1+M_{\e}^{\th_2}\|\vr_\e \|_{L^{2\g}(\O_\e)}^{1-\th_2})\|\vr_\e \|_{L^{2\g}(\O_\e)}^{\g},
\ea
where $s,t$ and $\th_2$ are the same as above.
\\Summing up the estimates for $I_1$ to $I_4$ implies
\be\label{2.18}
\|\vr_\e\|_{L^{2\g}(\O_\e)}^{2\g}   \leq  C\, \left(1+\|\vr_\e\|_{L^{2\g}(\O_\e)}^{\b_1(\g)}\right),
\ee
where
\be
\b_1(\g)=\mbox{max}\{2\g(1-\th_1),\g+(1-\th_3),\g+(1-\th_2),\g,\g+(1-\th_3)+2(1-\th_2)\}.\nn
\ee
Once $\g>1$ fixed, we can always find $s\rightarrow1+$ smaller than $\g$ \ \mbox{s.t.} $1-\th_2\rightarrow0+$ (for any $\g\in(1,+\infty)$). At the same time, $1-\th_3$ goes down with the increase of $\g$ and there exists $s_1\rightarrow2+$ \mbox{s.t.} $1-\th_3\rightarrow1-$ (for any $\g\rightarrow1+$) to obtain $\b_1(\g)<2\g$ in \eqref{2.18}.
\par Then we deduce
\be\label{2.19}
\|\vr_\e\|_{L^{2\g}(\O_\e)}  \leq  C,\quad C \ \mbox{is independent of}\,\ \e.
\ee
Moreover, combined with \eqref{2.4}, we have
\be\label{2.20}
 \|\vu_\e\|_{W_0^{1,2}(\O_\e; R^2)} \leq C,\quad C \ \mbox{is independent of}\,\ \e.
\ee
Let $[\tilde \vr_\e,\tilde \vu_\e]$ be the zero extension of $[\vr_\e, \vu_\e]$ in $\O$. Then by \eqref{2.19} and \eqref{2.20} we have
\be\label{2.21}
\|\tilde \vr_\e\|_{L^{2\g}(\O)}+   \|\tilde \vu_\e\|_{W_0^{1,2}(\O; R^2)}\leq  C.
\ee
Thus, up to a substraction of subsequence,
\be\label{2.22}
\tilde \vr_\e  \to \vr \ \mbox{weakly in} \ L^{2\g}(\O),\quad \tilde \vu_\e \to \vu \   \mbox{weakly in}\  W_{0}^{1,2}(\O; R^2).
\ee

We obtained the uniform estimate \eqref{1.22} and the weak convergence in \eqref{1.23}.

\medskip

\subsection{Equations in homogeneous domain}\label{Equations in a fixed domain}
Now, we show that the couple $[\tilde \vr_\e,\tilde \vu_\e]$ solves the momentum equations as \eqref{1.6} in $\O$ up to a small remainder.
\begin{lemma}\label{Lemma 2.3}
Under the assumptions in {\rm Theorem \ref{Theorem 1.6}},  there holds
 \be\label{2.23}
\dive(\tilde\vr_\e\tilde \vu_\e\otimes\tilde \vu_\e)+\nabla p(\tilde\vr_\e)=\dive\SSS(\nabla \tilde\vu_\e)+\tilde\vr_\e \vf + \vg+{\rm \mathbf{r}}_\e,\quad \mbox{in}\ \mathcal{D}'(\O; R^2),
 \ee
where the distribution ${\rm \mathbf{r}}_\e$ is small and satisfies:
\be\label{2.24}
|\langle \mathbf{r}_\e,\varphi\rangle_{\mathcal{D}'(\O; R^2),\mathcal{D}(\O; R^2)}|\leq C \e^{\de_1}(\|\nabla\varphi\|_{L^{t}(\O_\e; R^{2\times2})}+\|\varphi\|_{L^{\infty}(\O_\e; R^{2})}),
\ee
for any $\varphi\in C_c^{\infty}(\O;R^2)$ and $\de_1:=\frac{1}{2}\alpha-1>0, t\rightarrow 2+$, bounded constant $C>0$.
\end{lemma}

\begin{proof}[\bf Proof of Lemma \ref{Lemma 2.3}]

By the assumptions on the distribution and the size of the holes in \eqref{1.2}, there exists cut-off function $g_\e\in C_c^\infty(\O)$ satisfying $0\leq g_\e \leq 1$ and
\be\label{2.25}
g_\e =1 \ \mbox{on}\ \bigcup_{k\in K_\e}  T_{\e,k},\quad g_\e=0 \  \mbox{in} \ \O\setminus \bigcup_{k\in K_\e} \overline{B(x_{\e,k},b_0\e^\a)},\quad \|g_\e\|_{W^{1,2}(\O; R^2)}^{2}\leq 2C\e^{-2}|loga_\e|^{-1}.
\ee

Then for any $\varphi\in C_c^\infty (\O; R^2)$, we have
\ba\label{2.16}
I_\e&=\intO{\tilde\vr_\e\tilde \vu_\e\otimes\tilde \vu_\e:\nabla \varphi+ p(\tilde\vr_\e) \,\dive \varphi-\SSS(\nabla \tilde\vu_\e):\nabla \varphi + \tilde\vr_\e  \vf \cdot \varphi + \vg \cdot \varphi}\\
&=\int_{\O_\e}\Big(\tilde\vr_\e\tilde \vu_\e\otimes\tilde \vu_\e:\nabla ((1-g_\e)\varphi)+ p(\tilde\vr_\e)\, \dive ((1-g_\e)\varphi)-\SSS(\nabla \tilde\vu_\e):\nabla ((1-g_\e)\varphi) \\
&\quad+ \tilde\vr_\e \vf \cdot ((1-g_\e)\varphi) + \vg \cdot ((1-g_\e)\varphi)\Big)\,\dx + I_\e,\\ \nn
\ea
where we used the fact that $(1-g_\e)\varphi \in C_c^\infty (\O_\e; R^2)$ is a good test function for the momentum equations \eqref{1.6} in $\O_\e$, and the quantity $I_\e$ is of the form
\ba\label{26}
&I_\e:=\sum_{j=1}^4 I_{j,\e},
\ea
with
\ba
&I_{1,\e}:=  \intO{ \tilde \vr_\e \tilde \vu_\e\otimes \tilde \vu_\e:(g_\e\nabla \varphi) +\tilde \vr_\e \tilde \vu_\e\otimes \tilde \vu_\e:(\nabla g_\e\otimes  \varphi) },\\
&I_{2,\e}:= \intO{  p(\tilde \vr_\e)g_\e\dive \varphi+p(\tilde \vr_\e) \nabla g_\e\cdot \varphi},\\
&I_{3,\e}:= -\intO{ \SSS(\nabla \tilde \vu_\e):(g_\e\nabla \varphi) +\SSS(\nabla \tilde \vu_\e):(\nabla g_\e\otimes  \varphi)},\\
&I_{4,\e}:= \intO{  \tilde \vr_\e  \vf \cdot g_\e\varphi +\vg \cdot g_\e \varphi}.\nn
\ea
We now estimate $I_{j,\e}$ ($j=1,2,3,4$) one by one. For $I_{1,\e}$, direct calculation gives
\ba\label{2.27}
|I_{1,\e}|: &=|\intO{ \tilde \vr_\e \tilde \vu_\e\otimes \tilde \vu_\e:(g_\e\nabla \varphi) +\tilde \vr_\e \tilde \vu_\e\otimes \tilde \vu_\e:(\nabla g_\e\otimes  \varphi) }|\\
 &\leq  C\, \|\tilde\vr_\e\|_{L^{2\g}(\O;R)} \| \tilde\vu_\e\|_{L^{s'}(\O; R^2)}^2 \left( \|\nabla g_\e\|_{L^{2}(\O; R^2)}\|\varphi\|_{L^{t'}(\O; R^2)}+ \|g_\e\|_{L^{t'}(\O; R)}\|\nabla\varphi\|_{L^{2}(\O; R^{2\times2})}\right)\\
&\leq  C\, \|\nabla g_\e\|_{L^{2}(\O; R^2)}\|\nabla\varphi\|_{L^{2}(\O; R^{2\times2})}
\leq C \e^{\frac{1}{2}\alpha-1}\|\nabla\varphi\|_{L^{2}(\O; (R^{2\times2})}
\leq C \e^{\frac{1}{2}\alpha-1}\|\nabla\varphi\|_{L^{t}(\O; R^{2\times2})},
\ea
where
\be
\alpha>2,\quad \frac{2}{s'}+\frac{1}{t'}=(\frac{1}{2}-\frac{1}{2\g}),\quad 1<\g<\infty, \quad \mbox{and} \quad s',\ t'\rightarrow\infty-,\ t\rightarrow2+ \quad \mbox{as} \quad\g\rightarrow1+.\nn
\ee

\medskip

For $I_{2,\e}$ and $I_{3,\e}$, similar to the estimate for $I_{1,\e}$, we have
\ba\label{2.28}
|I_{2,\e}|:&= |\intO{  p(\tilde \vr_\e)g_\e\dive \varphi+p(\tilde \vr_\e) \nabla g_\e\cdot \varphi}|\\
 &\leq  C\, a\|\tilde \vr_\e\|_{L^{2\g}(\O;R)}^{\g}\left( \|g_\e\|_{L^{s}(\O; R)}\|\nabla\varphi\|_{L^{t}(\O; R^{2\times2})}+\|\nabla g_\e\|_{L^{2}(\O; R^2)}\|\varphi\|_{L^{\infty}(\O; R^2)}\right)\\
&\leq C \e^{\frac{1}{2}\alpha-1}(\|\nabla\varphi\|_{L^{t}(\O; R^{2\times2})}+\|\varphi\|_{L^{\infty}(\O; R^{2})}).
\ea

\ba\label{2.29}
|I_{3,\e}|:&=| -\intO{ \SSS(\nabla \tilde \vu_\e):(g_\e\nabla \varphi) +\SSS(\nabla \tilde \vu_\e):(\nabla g_\e\otimes  \varphi)}|\\
 &\leq C\, \|\nabla \tilde \vu_\e\|_{L^{2}(\O;R^{2\times2})}\left(\|g_\e\|_{L^{s}(\O; R)}\|\nabla\varphi\|_{L^{t}(\O; R^{2\times2})}+ \|\nabla g_\e\|_{L^{2}(\O; R^2)}\|\varphi\|_{L^{\infty}(\O; R^2)}\right)\\
 &\leq C \e^{\frac{1}{2}\alpha-1}(\|\nabla\varphi\|_{L^{t}(\O; R^{2\times2})}+\|\varphi\|_{L^{\infty}(\O; R^{2})}).
\ea
where
\be\label{2.30}
\alpha>2,\quad 1<\g<\infty, \quad \frac{1}{s}+\frac{1}{t}=\frac{1}{2}, \quad \mbox{and} \quad s\rightarrow\infty-, \ t\rightarrow2+.
\ee
For $I_{4,\e}$, the similar argument gives the following analogous estimate:
\ba\label{2.31}
|I_{4,\e}|:&= |\intO{  \tilde \vr_\e  \vf \cdot g_\e\varphi +\vg \cdot g_\e \varphi}|\\
&\leq C\, \|\varphi\|_{L^{\infty}(\O;R^2)}\left( \|\tilde\vr_\e\|_{L^{2\g}(\O;R)}\|\vf\|_{L^{\infty}(\O; R^2)}\| g_\e\|_{L^{(1-\frac{1}{2\g})^{-1}}(\O;R)}+\|\vg\|_{L^{\infty}(\O; R^2)}\|g_\e\|_{L^{1}(\O; R)}\right)\\
&\leq C\, \|\varphi\|_{L^{\infty}(\O;R^2)}\left( \|\tilde\vr_\e\|_{L^{2\g}(\O;R)}\|\vf\|_{L^{\infty}(\O; R^2)}\| g_\e\|_{L^{2}(\O; R)}+\|\vg\|_{L^{\infty}(\O; R^2)}\|g_\e\|_{L^{2}(\O; R)}\right)\\
&\leq C\, \|g_\e\|_{L^{2}(\O; R)}\|\varphi\|_{L^{\infty}(\O;R^2)}\leq C \e^{\frac{1}{2}\alpha-1}\|\varphi\|_{L^{\infty}(\O; R^{2})}.
\ea
\medskip
Summing up the estimates in \eqref{2.27},  \eqref{2.28}, \eqref{2.29} and \eqref{2.31}, we finally obtain
\be\label{2.32}
I_\e:=\sum_{j=1}^4 I_{j,\e} \leq C \e^{\frac{1}{2}\alpha-1}(\|\varphi\|_{L^{\infty}(\O; R^{2})}+\|\nabla\varphi\|_{L^{t}(\O; R^{2\times2})}),
\ee
where
\be\label{2.33}
 \de_1:=\frac{1}{2}\alpha-1>0,\quad t\rightarrow 2+.
\ee
Thus we completed the proof of Lemma \ref{Lemma 2.3}.
\end{proof}

\subsection{The limit equations}
In this section, we deduce the limit equation for the couple  $[\vr,\vu]$ obtained in \eqref{2.22} which represent a finite energy renormalized weak solution of \eqref{1.5} to \eqref{1.8} in $\O$. First of all, from \eqref{2.22}, we have the following convergence
\be\label{2.34}
\tilde \vr_\e  \to \vr \ \mbox{weakly in} \ L^{2\g}(\O),\quad \tilde \vu_\e \to \vu \   \mbox{weakly in}\  W_{0}^{1,2}(\O; R^2).
\ee
Applying  compact Sobolev embedding, we have
\be\label{2.35}
\tilde\vu_\e \to \vu \quad \mbox{strongly in}\quad L^q(\Omega; R^2) \quad \mbox{for any $1\leq q<\infty$}.
\ee
Then, the following weak convergence of nonlinear terms holds
\ba\label{2.36}
& \tilde\vr_\e \tilde\vu_\e \to \vr\vu \quad &&\mbox{weakly in}\quad L^q(\O; R^2) \quad \mbox{for any}\ \ 1< q<2\g,\\
& \tilde\vr_\e \tilde\vu_\e\otimes\tilde\vu \to \vr\vu\otimes\vu \quad &&\mbox{weakly in}\quad L^q(\O; R^{2\times2}) \quad \mbox{for any $1< q<2\g$}.
\ea

Then in Lemma 2.1 and \eqref{2.23}, passing with  $\e\to 0$ gives
\ba\label{2.37}
&\dive (\vr \vu) = 0,\\
&\dive (\vr \vu \otimes \vu) +\nabla  \overline {p(\vr)} =\dive \SSS(\nabla \vu)+\vr \vf+ \vg,
\ea
in the sense of distribution in ${\mathcal D}'(\O)$, where $\overline {p(\vr)}$ is the weak limit of $p(\tilde \vr_\e)$ in $L^{2}(\O)$. Furthermore, by Lemma \ref{Lemma 1.4}, $[\vr,\vu]$ satisfies the renomalized equation
\be\label{2.38}
\dive \big( b(\vr)\vu \big) + \big( \vr b'(\vr) - b(\vr) \big) \dive \vu = 0, \quad \mbox{in}\ \mathcal{D}'( R^2),
\ee
where $b\in C^0([0,\infty))\cap C^1((0,\infty))$ satisfies \eqref{1.18}-\eqref{1.19}. To finish the proof of Theorem \ref{Theorem 1.6}, it suffices to show $\overline {p(\vr)} = p(\vr)$. Thus we obtain  in the following section.

\smallskip

\subsection{Convergence of pressure term - end of the proof}\label{Convergence of pressure term - end of the proof}

Here, we introduce $p(\vr)-(\frac{4\mu}{3}+\eta)\dive \vu$ called effective viscous flux, which possesses some weak compactness property specified in the following lemma, which take up great significance in the existence theory of weak solutions for the compressible Navier-Stokes equations.

\begin{lemma}\label{Lemma 2.4}
Up to a substraction of subsequence, there holds for any $\psi\in C_c^\infty(\Omega)$:
\be\label{2.39}
\lim_{\e\to 0}\intO{\psi \left(p(\tilde\vr_\e)-(\frac{4\mu}{3}+\eta)\dive \tilde\vu_\e\right)\tilde\vr_\e}=\intO{\psi\left(\overline{p(\vr)}-(\frac{4\mu}{3}+\eta)\dive \vu\right)\vr}.
\ee
\end{lemma}
\begin{proof}[\bf Proof of Lemma \ref{Lemma 2.4}]
The main idea is to give proper test functions via taking advantage of Fourier multiplier and Riesz type operators. Referring to Section 1.3.7.2 in \cite{N-book} or Section 10.16 in \cite{F-N-book} for the definitions and properties used here of Fourier multiplier and Riesz operators,  we choose proper test functions defined by
\be\label{2.40}
\psi \nabla \Delta^{-1}(1_{\O}\tilde\vr_\e), \quad  \psi \nabla \Delta^{-1}(1_{\O}\vr),
\ee
where $\psi\in C_c^\infty(\O)$ and $\Delta^{-1}$ is the Fourier multiplier on $ R^2$ with symbol $-\frac{1}{|\xi|^2}$.

Observing that
$$
\nabla \nabla \Delta^{-1}=\left(\mathcal{R}_{i,j}\right)_{1\leq i,j\leq 2}
$$
are the classical Riesz operators, then for any $f\in L^q( R^2),~1<q<\infty$, we have
\be
\|\nabla\nabla \Delta^{-1}(f)\|_{L^q( R^2; R^{2\times2})}\leq C\, \|f\|_{L^q( R^2)}.\nn
\ee
By virtue of $(1_{\O}\tilde\vr_\e)\in L^{2\g}( R^2; R)\bigcap L^{1+}( R^2; R)$, owing to $2<2\g<\infty, 1<1+<\infty$, we have
\be
\|\nabla\nabla \Delta^{-1}(1_{\O}\tilde\vr_\e)\|_{L^{2\g}( R^2; R^{2\times2})}\leq C\, \|1_{\O}\tilde\vr_\e\|_{L^{2\g}(R^2; R)}\leq C.\nn
\ee
\be
\|\nabla\nabla \Delta^{-1}(1_{\O}\tilde\vr_\e)\|_{L^{1+}( R^2; R^{2\times2})}\leq C\, \|1_{\O}\tilde\vr_\e\|_{L^{1+}(R^2; R)}\leq C.\nn
\ee
That means $\nabla\nabla \Delta^{-1}(1_{\O}\tilde\vr_\e)\in L^{2\g}( R^2; R^{2\times2})\bigcap L^{1+}( R^2; R^{2\times2})$. Since $2\g>2+>d=2$, by the embedding theorem in homogeneous Sobolev spaces (see Theorem 1.55 and Theorem 1.57 in \cite{N-book} or Theorem 10.25 and Theorem 10.26 in \cite{F-N-book}), we have
\be\label{2.41}
\nabla \Delta^{-1}(1_{\O}\tilde\vr_\e)\in W^{1,2\g}( R^2; R^{2})\hookrightarrow L^{\infty}( R^2; R^{2})
\ee
which means
\be\label{2.42}
\|\nabla \Delta^{-1}(1_{\O}\tilde\vr_\e)\|_{L^{\infty}( R^2; R^{2})}\leq \|\nabla\nabla \Delta^{-1}(1_{\O}\tilde\vr_\e)\|_{L^{2\g}( R^2; R^{2\times2})}
\leq C\, \|1_{\O}\tilde\vr_\e\|_{L^{2\g}(R^2; R)}\leq C.
\ee
 Again by the embedding theorem in homogeneous Sobolev spaces, we have for any $f\in L^q( R^2), \ \supp\, f\subset \O$:
\ba\label{2.43}
&\|\nabla \Delta^{-1}(f)\|_{L^{q^*}( R^2; R^2)}\leq C\, \| f \|_{L^q( R^2)},\quad \frac{1}{q^*}=\frac{1}{q}-\frac{1}{2}, \ \mbox{if}\  1<q<2,\\
\ea
since $1_{\O}\tilde\vr_\e\in L^{1+}( R^2; R)\bigcap L^{2^\g}( R^2; R)$, by interpolation theorem between Lebesgue spaces, then we have $1_{\O}\tilde\vr_\e\in L^{p}( R^2; R),\ 1+<p<2\g$ and
\be
\|\nabla \Delta^{-1}(1_{\O}\tilde\vr_\e)\|_{L^{q^*}( R^2; R^{2})}\leq C\, \|1_{\O}\tilde\vr_\e\|_{L^{q}(R^2; R)}\leq C,\nn
\ee
where $\frac{1}{q^*}=\frac{1}{q}-\frac{1}{2}$, $1+<q<2$, $2+<q^*<\infty-$. Combined with \eqref{2.41}, we obtain
\be
\nabla \Delta^{-1}(1_{\O}\tilde\vr_\e)\in L^{\tilde{q}}( R^2; R^{2}), \quad 2+\leq\tilde{q}\leq\infty.\nn
\ee

Then by the uniform estimate for $\tilde \vr_\e$ and its weak limit $\vr$ in \eqref{2.22} and the fact $2\g>2+$ under our assumption $\g>1$,

\ba\label{2.44}
&\|\nabla \left(\psi \nabla \Delta^{-1}(1_{\O}\tilde\vr_\e)\right)\|_{L^{2\g}(\O; R^{2\times2})} + \|\nabla \left(\psi \nabla \Delta^{-1}(1_{\O}\vr)\right)\|_{L^{2\g}(\O; R^{2\times2})}\leq C.
\ea

Since $\de_1:=\frac{1}{2}\alpha-1>0$ in Lemma \ref{Lemma 2.3}, thus, \eqref{2.24} and \eqref{2.42} implies
\ba\label{est-test-flux4}
&|\langle r_\e,\psi \nabla \Delta^{-1}(1_{\O}\tilde\vr_\e)\rangle_{\mathcal{D}'(\O; R^2),\mathcal{D}(\O; R^2)}|\\
&\leq C\, \e^{\de_1} \left(\|\nabla \left(\psi \nabla \Delta^{-1}(1_{\O}\tilde\vr_\e)\right)\|_{L^{t}(\O; R^{2\times2})} + \|\psi \nabla \Delta^{-1}(1_{\O}\tilde\vr_\e)\|_{L^{\infty}(\O; R^2)}\right)\\
&\leq C\, \e^{\de_1} \left(\|\nabla \left(\psi \nabla \Delta^{-1}(1_{\O}\tilde\vr_\e)\right)\|_{L^{2\g}(\O; R^{2\times2})} + \|\psi \nabla \Delta^{-1}(1_{\O}\tilde\vr_\e)\|_{W^{1,2\g}(\O; R^2)}\right)\\
&\leq C\,\e^{\de_1},\ t\rightarrow 2{+},\nn
\ea
which goes to zero as $\e\to 0$.

\smallskip

Then we chose $\psi \nabla \Delta^{-1}(1_{\Omega}\tilde\vr_\e)$ as a test function in the weak formulation of equation \eqref{2.23} and pass $\e\to 0$. Moreover, we choose $\psi \nabla \Delta^{-1}(1_{\Omega}\vr)$ as a test function in the weak formulation of $\eqref{2.37}_2$. Comparing the results of theses two operations, through long and straightforward calculations, we finally get
\ba\label{2.45}
I:&=\lim_{\e\to 0}\intO{\psi \left(p(\tilde\vr_\e)-(\frac{4\mu}{3}+\eta)\Div \tilde\vu_\e\right)\tilde\vr_\e}-\intO{\psi\left(\overline{p(\vr)}-(\frac{4\mu}{3}+\eta)\Div \vu\right)\vr}\\
&=\lim_{\e\to 0}\intO{\tilde\vr_\e\tilde\vu_\e^i\tilde\vu_\e^j \psi \mathcal{R}_{i,j}(1_{\Omega} \tilde\vr_\e)}-\intO{\vr \vu^i\vu^j \psi \mathcal{R}_{i,j}(1_{\Omega} \vr)}.
\ea

In addition, choosing $1_{\Omega} \Div \Delta^{-1}(\psi\tilde\vr_\e \tilde\vu_\e)$ as a test function in the weak formulation of Lemma 2.1 with $b(\vr)=\vr$ and  $1_{\Omega} \Div  \Delta^{-1}(\psi\vr \vu)$ as a test function in the weak formulation of $\eqref{2.37}_1$ yields
\be\label{2.46}
\intO{1_{\Omega}\tilde\vr_\e \tilde\vu_\e^i  \mathcal{R}_{i,j}(\psi\tilde\vr_\e \tilde\vu_\e) }=0,\quad \intO{1_{\Omega}\vr \vu^i  \mathcal{R}_{i,j}(\psi \vr \vu)}=0.
\ee

Substituting (\ref{2.46}) into (\ref{2.45}) generates
\ba\label{2.47}
&I=\lim_{\e\to 0}\intO{\tilde\vu_\e^i\Big(\tilde\vr_\e\tilde\vu_\e^j \psi \mathcal{R}_{i,j}(1_{\Omega} \tilde\vr_\e)-1_{\Omega}\tilde\vr_\e   \mathcal{R}_{i,j}(\psi\tilde\vr_\e \tilde\vu_\e)\Big)}\\
&\qquad\qquad-\intO{ \vu^i\Big(\vr \vu^j \psi \mathcal{R}_{i,j}(1_{\Omega} \vr)-1_{\Omega}\vr  \mathcal{R}_{i,j}(\psi \vr \vu)\Big)}.
\ea

Now, we introduce the following Lemma.
\begin{lemma}\label{Lemma 2.5}Let $1<p,q<\infty$ satisfy $$\frac{1}{r}:=\frac{1}{p}+\frac{1}{q}<1.$$
Suppose
\[
u_\e \to u \quad\mbox{weakly in}\quad L^p( R^2),\quad v_\e \to v \quad\mbox{weakly in}\quad L^q( R^2),\ \mbox{as $\e \to 0$}.
\]
Then for any  $1\leq i,j\leq 2$:
\[
u_\e \mathcal{R}_{i,j}(v_\e)-v_\e \mathcal{R}_{i,j}(u_\e) \to u \mathcal{R}_{i,j}(v)-v \mathcal{R}_{i,j}(u) \quad\mbox{weakly in}\quad L^r( R^2).
\]
\end{lemma}

Critically, applying lemma \ref{Lemma 2.5} (refer to \cite[Lemma 3.4]{FNP} for the proof) to \eqref{2.47} to obtain $I\rightarrow 0 (as \ \e\rightarrow0)$ in \eqref{2.45}, then the convergence result \eqref{2.39} can be deduced.

\end{proof}

\medskip

A direct consequence of the compactness of the effective viscous flux is as follows.
\begin{lemma}\label{Lemma 2.6} We denote $\overline {p(\vr)\vr}$ as the weak limit of $p(\tilde \vr_\e)\tilde \vr_\e$ in $L^{\frac{2\g}{\g+1}}(\O)$. Then we have $\overline {p(\vr)\vr}=\overline {p(\vr)}\vr$.

\end{lemma}

\begin{proof}[\bf Proof of Lemma \ref{Lemma 2.6}]
In the beginning, we have
$$ 2\g-(\g+1)=\g -1 >0 . $$
Then by \eqref{2.21}, we obtain
$$
p(\tilde \vr_\e)\tilde \vr_\e \to \overline {p(\vr)\vr} \quad \mbox{weakly in}\quad L^{\frac{2\g}{\g+1}}(\O).
$$

Taking $b(s)=s\log s$ in the renormalized equations in Lemma 2.1 and \eqref{2.38} yields
\be\label{2.48}
\dive\big((\tilde\vr_\e\log \tilde\vr_\e) \tilde\vu_\e\big)+\tilde\vr_\e \dive \tilde\vu_\e=0,\quad \dive\big((\vr\log \vr) \vu\big)+\vr \dive \vu=0,\ \mbox{in}\ \mathcal{D}'(\O).
\ee
Passing $\e\to 0$ in the first equation of \eqref{2.48} gives
\be\label{2.49}
\dive\big(\overline{(\vr\log \vr)}\, \vu\big)+\overline{\vr \dive \vu}=0, \ \mbox{in}\ \mathcal{D}'(\O),
\ee
where we used the strong convergence of the velocity in \eqref{2.35} and
\ba\label{2.50}
&\tilde\vr_\e\log \tilde\vr_\e \to \overline{\vr\log \vr} \quad \mbox{weakly in}\quad L^{q}(\O)\ \mbox{for any $q<2\g$},\\
&\tilde\vr_\e \dive \tilde\vu_\e \to \overline{\vr \dive \vu} \quad \mbox{weakly in}\quad L^{\frac{2\g}{\g+1}}(\O).
\ea

Then for any $\psi\in C_c^\infty(\Omega)$, \eqref{2.49} and \eqref{2.50} implies
\be\label{2.51}
\eqref{2.39}_{left}=\lim_{\e\to 0}\intO{\psi \left(p(\tilde\vr_\e)-(\frac{4\mu}{3}+\eta)\dive \tilde\vu_\e\right)\tilde\vr_\e}=\intO{\psi \overline {p(\vr)\vr} -(\frac{4\mu}{3}+\eta) \overline{(\vr\log \vr)}\, \vu \cdot \nabla \psi}.
\ee

Utilizing the second equation in \eqref{2.48}, we obtain
\be\label{2.52}
\eqref{2.39}_{right}=\intO{\psi\left(\overline{p(\vr)}-(\frac{4\mu}{3}+\eta)\dive \vu\right)\vr}=\intO{\psi \overline{p(\vr)}\vr -(\frac{4\mu}{3}+\eta) (\vr\log \vr)\, \vu \cdot \nabla \psi}.
\ee

Assume the test functions $\{\psi_n\}_{n\in \Z_+} \subset C_c^\infty(\Omega)$ such that
$$
 \psi_n(x)=0 \ \mbox{if} \ d(x,\partial \Omega)<\frac{1}{n},\quad \psi_n(x)=1 \ \mbox{if} \ d(x,\partial \O)>\frac{2}{n},\quad \|\nabla  \psi_n\|_{L^\infty(\O; R^2)}\leq 2n.
$$
Then for any $q\in [1,\infty]$:
\[
\|1-\psi_n\|_{L^q(\O)}\leq C\, n^{-\frac{1}{q}},\quad \|\nabla \psi_n\|_{L^q(\O; R^2)}\leq C\, n^{1-\frac{1}{q}},
\]
and consequently
\[
\|d(x,\d \O)\nabla \psi_n\|_{L^q(\O; R^2)}\leq C\, n^{-\frac{1}{q}}.
\]

Here, by Hardy's inequality, the velocity $\vu\in W_0^{1,2}(\O; R^2)$ implies
\[
\|[d(x,\partial \O)]^{-1} \vu\|_{L^2(\O; R^2)}\leq C\|\vu\|_{W^{1,2}(\O; R^2)}\leq C.
\]
Therefore,
\ba\label{2.53}
&\intO{\nabla \psi_n \cdot \overline{(\vr\log\vr)}\vu} \\
&\leq \|d(x,\partial \O)\nabla \psi_n\|_{L^{10}(\O; R^2)}\|\overline{(\vr\log\vr)}\|_{L^{5/2}(\O)}\|[d(x,\partial \Omega)]^{-1}\vu\|_{L^2(\O; R^2)}\\
&\leq C\ n^{-1/10}.
\ea
Similarly,
\be\label{2.54}
\intO{\nabla \psi_n \cdot (\vr\log\vr)\vu} \leq C \ n^{-1/10}.
\ee

Take $\psi=\psi_n$ in \eqref{2.39} and pass to the limit $n\to \infty$. By using \eqref{2.51}-\eqref{2.54}, we deduce
\be\label{2.55}
\intO{\overline{p(\vr)\vr}-\overline{p(\vr)}\vr}=0.
\ee

By the strict monotonicity of the mapping $\vr \mapsto p(\vr)$, applying Theorem 10.19 in \cite{F-N-book} or Lemma 3.35 in \cite{N-book} implies
\[
\overline{p(\vr)\vr} \geq \overline{p(\vr)}\vr,\quad \mbox{a.e. in}\quad \O.
\]
Together with \eqref{2.55}, we deduce
\[
\overline{p(\vr)\vr} = \overline{p(\vr)}\vr,\quad \mbox{a.e. in}\quad \O.
\]

Thus we complete the proof of Lemma \ref{Lemma 2.6}.

\end{proof}

By virtue of the monotonicity of $p(\cdot)$, and using Theorem 10.19 in \cite{F-N-book} again, we obtain $\overline{p(\vr)}=p(\vr)$. Hence, we finish the proof of Theorem \ref{Theorem 1.6}.

For convenience, we recall Theorem 10.19 in \cite{F-N-book}: Let $I\subset R$ be an interval, $Q\subset R^d$ be a domain, $P$ and $G$ be non-decreasing functions in $C(I)$. Let $\{\vr_n\}_{n\in \N}$ be a sequence in $L^1(Q;I)$ such that
$$
P(\vr_n) \to \overline{P(\vr)}, \quad G(\vr_n) \to \overline{G(\vr)}, \quad P(\vr_n)G(\vr_n) \to \overline{P(\vr)G(\vr)}, \quad \mbox{weakly in $L^1(Q)$}.
$$
Then the following properties hold:
\begin{itemize}

\item[(i).] $\overline{P(\vr)}\ \overline{G(\vr)}\leq \overline{P(\vr)G(\vr)}.$
\item[(ii).] If $P\in C(R), \ G\in C(R),\ G(R)=R$, $G$ is strictly increasing, and $\overline{P(\vr)}\ \overline{G(\vr)} = \overline{P(\vr)G(\vr)}$, then $\overline{P(\vr)}=P\circ G^{-1} \overline{G(\vr)} $.  If, in particular, $G(z)=z$ be the identity function, there holds $\overline{P(\vr)}=P(\vr)$.

\end{itemize}

{\bf Acknowledgement}: The research of \v{S}.N.  leading to these results has received funding from the Czech Sciences Foundation (GA\v CR), GA19-04243S and in the framework of RVO: 67985840.  The authors are grateful to Yong Lu for fruitful discussions.



\end{document}